\documentclass{article}
\usepackage[utf8]{inputenc}

\title{Model structures on the category of complexes of quiver representations}

\author{Payam Bahiraei}
\date{27 February 2021}

\input xy
\xyoption{all}
\newcommand{\rt}{\rightarrow}
\newcommand{\lrt}{\longrightarrow}

\newcommand{\CA}{\mathcal{A} }
\newcommand{\CC}{\mathcal{C} }

\newcommand{\CE}{\mathcal{E}}
\newcommand{\CF}{\mathcal{F} }

\newcommand{\CQ}{\mathcal{Q} }

\newcommand{\CX}{\mathcal{X} }
\newcommand{\CY}{\mathcal{Y} }

\newcommand{\Mod}{{\rm{Mod\mbox{-}}}}

\newcommand{\mmod}{{\rm{{mod\mbox{-}}}}}

\newcommand{\Prj}{{\rm{Prj}\mbox{-}}}

\newcommand{\QR}{{{\rm Rep}{(\CQ, R)}}}

\newcommand{\ac}{{\rm{ac}}}

\newcommand{\Coker}{{\rm{Coker}}}
\newcommand{\Ker}{{\rm{Ker}}}

\newcommand{\Ext}{{\rm{Ext}}}

\newcommand{\Ho}{{\rm{Ho}}}

\newtheorem{theorem}{Theorem}[section]
\newtheorem{corollary}[theorem]{Corollary}
\newtheorem{lemma}[theorem]{Lemma}

\newtheorem{proposition}[theorem]{Proposition}
\newtheorem{definition}[theorem]{Definition}
\newtheorem{example}[theorem]{Example}

\newtheorem{notation}[theorem]{Notation}
\newtheorem{remark}[theorem]{Remark}
\newtheorem{proof}{Proof}
\begin{document}

\maketitle
\begin{abstract}
 In this paper, we study the category $C(Rep(\mathcal{Q}, \mathcal{A}))$ of complexes of representations of quiver $\mathcal{Q}$ with values in an abelian category $\mathcal{A}$.  We develop a method
for constructing some model structures on $C(Rep(\mathcal{Q}, \mathcal{A}))$ based on componentwise notion. Moreover we also show that these model structures
are monoidal. As an application of these model structures we introduce some descriptions of the derived category of complexes of representations of $\mathcal{Q}$ in $\Mod R$. 
\end{abstract}

\section{Introduction}\label{sec1}

The notion of cotorsion pairs (or cotorsion theory) was invented by \cite{Sal79} in the category of abelian groups and was rediscovered by Enochs and coauthors in the 1990's. In short, a cotorsion pair in an abelian category $\mathcal{A}$ is a pair $(\CF,\CC)$ of classes of object of $\mathcal{A}$ each of which is the orthogonal complement of the other with respect to the $\Ext$ functor. In recent years we have seen that the study of cotorsion pairs is especially relevant to study of covers and envelops, particularly in the proof of the flat cover conjecture \cite{BBE}.

There is another usage of cotorsion pairs in abelian model structures introduced by Hovey in \cite{Hov02}. Hovey noticed that a Quillen model structure on any abelian category $\mathcal{A}$ is equivalent to two complete cotorsion pairs in $\mathcal{A}$ which are compatible in a precise way. These cotorsion pairs are called Hovey pair.  In \cite{Gil04}, Gillespie began the study of when a cotorsion pair in abelian category $\mathcal{A}$, induces two compatible cotorsion pairs in $C(\mathcal{A})$, the category of unbounded complexes of $\mathcal{A}$. He applied
Hovey's approach to define new and interesting abelian model structure on $C(R)$, which is monoidal in the sense of \cite{Hov99} where $R$ is an associative unitary commutative ring. This approach was also followed in \cite{Gil06, Gil08, CEG, EER08, EAPT, EEI} in order to find new classes which give rise to new abelian model structure in certain abelian categories of unbounded complexes.

The representation theory of quivers is probably one of the most fruitful parts of
modern representation theory. By now, a number of remarkable connections to other
algebraic topics have been discovered, in particular to Lie algebras, Hall
algebras and quantum groups and more recently to cluster algebras.
Let $Rep(\mathcal{Q}, \mathcal{A})$ be the category of $\mathcal{A}$-valued representations of quiver $\mathcal{Q}$, where $\mathcal{A}$ is an abelian category. There is an interesting question in \cite{HJ1} 'How homological properties in $\mathcal{A}$ carry over to $Rep(\mathcal{Q}, \mathcal{A})$?' In series of papers Enochs, et al presented descriptions for projective, injective and flat object of $Rep(\mathcal{Q}, \mathcal{A})$ with respect to their local properties. See \cite{EE, EER09, EOT}. In \cite{EHHS}, Eshraghi, et al. studied the cotorsion pair in $Rep(\mathcal{Q},R)$. They showed that in a certain conditions, a complete cotorsion pair in $\Mod R$ can be given a complete cotorsion pair in $\QR$ and vice versa. Recently in \cite{HJ1} Holm and J\o rgensen extend this result about module-valued quiver representations to general $\mathcal{M}$-valued representations where $\mathcal{M}$ is an abelian category. Now let $\mathcal{Q}$ be an acyclic finite quiver and set $\mathcal{M}=C(\mathcal{A})$, the category of complexes of $\mathcal{A}$, then we have such cotorsion pairs in the category of $Rep(\mathcal{Q},C(\mathcal{A}))$, since $C(Rep(\mathcal{Q},\mathcal{A}))=Rep(\mathcal{Q},C(\mathcal{A}))$. We follow these results and focus on the study of Hovey pair in the category of complexes of representations of quivers. So we start by a Hovey pair in $C(\mathcal{A})$ and induce two Hovey pairs in $C(Rep(\mathcal{Q}, \mathcal{A}))$, see Theorem \ref{prop 32}.

Therefore based on well-known model structures on $C(\mathcal{A})$ we construct new model structures on $C(Rep(\mathcal{Q}, \mathcal{A}))$. Moreover we also show that these model structures are monoidal when the model structures are monoidal in $C(\mathcal{A})$. These model structures are more related with a componentwise notion, so we call them componetwise model structures. In this case, the derived category, $D(Rep(\mathcal{Q}, \mathcal{A}))$, occurs as its homotopy category.

In this work,  we introduce two descriptions for $D(\QR)$ the derived category of complexes of representations of quivers by $R$-modules (usually abbreviated to $D(\mathcal{Q})$) in terms of componentwise notion. In fact we characterize the homotopy relation of componentwise projective model structure and show that if $\mathcal{Q}$ is an acyclic finite quiver then we have the following equivalence
\begin{equation}
\label{eq1}
D(\mathcal{Q})\cong K({dgPrj}^{op}\mbox{-}\mathcal{Q})
\end{equation}
where ${dgPrj}^{op}\mbox{-}\CQ$  is defined in section \ref{section 4}. This equivalence is obtained under the canonical functor $K(\mathcal{Q})\longrightarrow D(\mathcal{Q})$. Hence we introduce a subcategory, differ from subcategory of DG-projective complexes of $K(\mathcal{Q})$ such that equivalent to $D(\mathcal{Q})$ under the canonical functor $K(\mathcal{Q})\longrightarrow D(\mathcal{Q})$. We also establish another description for $D(\CQ)$ as a quotient of homotopy category of representations of projective complexes based on componentwise notion, see corollary \ref{corollary 5}.

The organization of this paper is as follows. In section \ref{preliminaries}, we recall some generalities on model structures and facts needed. Section \ref{section 3} we focus  on componentwise model structures on $C(Rep(\CQ,\mathcal{A}))$. As an application of these model structures we introduce two descriptions for $D(\QR)$ the derived category of complexes of representations of quivers by $R$-modules
in section \ref{section 4}

\section{Preliminaries}
\label{preliminaries}
\subsection{The homotopy category of complexes}
Let $\mathcal{A}$ be an additive category and $C(\mathcal{A})$ denote the category of complexes over $\mathcal{A}$.
Morphisms $f, g : X^\bullet \rightarrow Y^\bullet$ in the category $C(\mathcal{A})$
of complexes are called homotopic, denoted by $f \sim g$, if there exists a family
$(s^n)_{n\in Z}$ of morphisms $s^n : X^n \rightarrow Y^{n-1}$ in $\mathcal{A}$, satisfying $f^n - g^n = d_Y^{
n-1}s^n +
s^{n+1}d_X^
n$ for all $n\in Z$. It is easy to check that $\sim$ is an equivalence relation. The classical homotopy category of $\mathcal{A}$, denoted by $K(\mathcal{A})$, has the same objects as $C(\mathcal{A})$ but morphisms are the homotopy classes of morphisms of complexes.

There is also another interpretation of the homotopy category of complexes. In fact in view of \cite{Hap88}, $K(\mathcal{A})$ is the stable category of Frobenius category $(C(\mathcal{A}),\mathcal{S})$, where
 $\mathcal{S}$ is the set of all exact sequences
$0^\bullet \rightarrow  X^\bullet \rightarrow Y^\bullet \rightarrow  Z^\bullet \rightarrow  0^\bullet$ in $\mathcal{A}$ such that the exact sequences $0 \rightarrow  X^i \rightarrow Y^i \rightarrow  Z^i \rightarrow  0^i$
are split exact. In this case, we can see that if $f, g : X^\bullet \rightarrow Y^\bullet$ are two morphisms in $C(\mathcal{A})$ then  $f \sim g$ if and only if $f-g$ factors over an $\mathcal{S}$-injective object $I(X^\bullet)$, where $I(X^\bullet)$ is the complex $\oplus_{i\in Z} \overline{X^i}[i]$. Here $\overline{X}$ is the complex $X^0=X^{-1}=X$ with the identity map and zero elsewhere.

If $R$ is an associative ring with identity and set $\mathcal{A}=\Mod R$, the category of left $R$-modules, then for any $R$-module $M$, the complex $\overline{M}$ is an 
$\mathcal{S}$-projective and $\mathcal{S}$-injective object of an exact category $C(\Mod R)$ (usually abbreviated to $C(R)$) with set of exact sequences $\mathcal{S}$ as above. In general,
for every projective module $P$, the complex
$$ \cdots \rightarrow 0 \rightarrow P \rightarrow P \rightarrow 0 \rightarrow \cdots $$
is projective. We can also say that any projective complex can be written uniquely as coproduct of such complexes. Dually, if $I$ is an injective module, the complex
$$ \cdots \rightarrow 0 \rightarrow I \rightarrow I \rightarrow 0 \rightarrow \cdots $$
is injective. Furthermore, up to isomorphism, any injective complex is a direct product of such complexes. Note that this direct product is in fact a direct sum.

\subsection{The category of representation of quiver:}
Let $\mathcal{Q}$ be a quiver (a directed graph). The sets of vertices and arrows are denoted by $V(\mathcal{Q})$ and $E(\mathcal{Q})$ respectively and are usually abbreviated to $V$ and $E$. An arrow of a quiver from a vertex $v_1$ to a vertex $v_2$ is denoted by $a:v_1\rightarrow v_2$. In this case we write $s(a)=v_1$ the initial (source) vertex and $t(a)=v_2$ the terminal (target) vertex. A path $p$ of a quiver $\mathcal{Q}$ is a sequence of arrows $a_n \cdots a_2 a_1$ with $t(a_i)=s(a_{i+1})$. 

\vspace{1mm}
A quiver $\mathcal{Q}$ is said to be finite if $V$ and $E$ are finite sets. A path of length $l\geq 1$ is called cycle whenever its source and target coincide. A quiver is called acyclic if it contains no cycles.

\vspace{1mm}

Let $\mathcal{A}$ be an abelian category. 
A representation $\mathcal{X}$ by objects of $\mathcal{A}$ of a given quiver $\mathcal{Q}$ is a covariant functor $\mathcal{X}:\CQ \longrightarrow \mathcal{A}$, so a representation is determined by giving object $\mathcal{X}_v\in\mathcal{A}$ to each vertex $v$ of $\mathcal{Q}$ and a morphism $\mathcal{X}(a):\mathcal{X}_v\rightarrow \mathcal{X}_w$ in $\mathcal{A}$ to each arrow $a:v\rightarrow w$ of $\mathcal{Q}$. A morphism $\varphi$ between two representations $\CX,\CY$ is just a natural transformation between $\CX,\CY$ as a functor. Indeed, $\varphi$ is a family $(\varphi_v)_{v\in V}$ of maps $(f_v:\CX_v \lrt \CY_v)_{v\in V}$ such that for each arrow $a:v \lrt w$, we have $\CY(a)\varphi_v = \varphi_w \CX(a)$ or, equivalently, the following square is commutative:
$$\xymatrix{ \CX_v \ar[r]^{\CX(a)} \ar[d]^{\varphi_v} & \CX_w \ar[d]^{\varphi_w} \\
\CY_v \ar[r]^{\CY(a)} & \CY_w }$$

We denoted by $Rep(\mathcal{Q},\mathcal{A})$ the category of all representations of $\mathcal{Q}$ by objects of $\mathcal{A}$. It can be seen that this category is an abelian category. If $R$ is an associative ring with identity we write $text{Rep}(\CQ,R)$ (resp. $text{rep}(\CQ,R)$) instead of $Rep(\mathcal{Q}, \Mod R)$ (resp. $Rep(\CQ, \mmod R)$).
 It is known that the category $text{Rep}(\CQ,R)$ is equivalent to the category of modules over the path algebra $R\CQ$, whenever $\mathcal{Q}$ is a finite quiver.

For any vertex $v\in V$ of quiver $\CQ=(V,E)$, let $e^v_{\mathcal{A}}: Rep(\mathcal{Q},\mathcal{A}) \lrt \mathcal{A}$ be the evaluation functor defined by $e^v_{\mathcal{A}}(\CX)=\CX_v$, for any $\CX\in Rep(\mathcal{Q},\mathcal{A})$. It is proved in \cite{EH} that $e^v_{\mathcal{A}}$ has a right adjoint $e^v_{\rho,\mathcal{A}}:\mathcal{A} \lrt Rep(\mathcal{Q},\mathcal{A})$ given by $e^v_{\rho,\mathcal{A}}(M)_w=\prod_{Q(w,v)}M$ for the object $M\in \mathcal{A}$, where $Q(w,v)$ denotes the set of paths starting in $w$ and terminating in $v$, and by the natural projection $e^v_{\rho,\mathcal{A}}(M)_a:\prod_{Q(w_1,v)}M \lrt \prod_{Q(w_2,v)}M$ for an arrow $a:w_1 \rightarrow w_2$. Moreover, it is shown that $e^v_{\mathcal{A}}$ admits also a left adjoint $e^v_{\lambda,\mathcal{A}}$, defined by $e^v_{\lambda,\mathcal{A}}(M)_w=\bigoplus_{Q(v,w)}M$.

\subsection{Model structures and Hovey pairs.}
Model categories were first introduced by Quillen \cite{Qui67}. Let $\mathcal{C}$ be a category. A model structure on $\mathcal{C}$ is a triple $(Cof, W, Fib)$
of classes of morphisms, called cofibrations, weak equivalences and fibrations, respectively,
such that satisfying certain axioms. The definition then was modified
by some authors. The one that is commonly used nowadays is due to Hovey \cite{Hov02}. Hovey
discovered that the existence of a model structure on any abelian category $\mathcal{A}$ is equivalent to
the existence of two complete cotorsion pairs in $\mathcal{A}$ which are compatible in a precise way.

Recall that a pair $(\mathcal{F}, \mathcal{C})$ of classes of objects of $\mathcal{A}$ is said to be a cotorsion pair if $\mathcal{F}^\perp=\mathcal{C}$ and $\mathcal{F} = {}^\perp\mathcal{C}$, where the left and right orthogonals are defined as follows 
\[{}^\perp\mathcal{C}:=\{A \in \mathcal{A}\ | \ \Ext^1_{\mathcal{A}}(A,Y)=0, \ {\rm{for \ all}} \ Y \in \mathcal{C} \}\]
and
\[\mathcal{F}^\perp:=\{A \in \mathcal{A}\ | \ \Ext^1_{\mathcal{A}}(W,A)=0, \ {\rm{for \ all}} \ W \in \mathcal{F} \}.\] 
A cotorsion pair  $(\mathcal{F}, \mathcal{C})$ is called complete if for every $A \in \mathcal{A}$ there exist exact sequences
\[0 \rt Y \rt W \rt A \rt 0 \ \ \ {\rm and} \ \ \ 0 \rt A  \rt Y' \rt W' \rt 0,\]
where $W, W'\in \mathcal{F}$ and $Y, Y' \in  \mathcal{C}$.

The advantage of the Hovey's theorem is that we can construct a model structure on abelian category $\mathcal{A}$ determined by three classes of objects, called cofibrant, trivial and fibrant objects. 
If $\mathcal{A}$ is a model category, then it has an initial object $\emptyset$ and terminal object $*$. An object $W\in \mathcal{A}$ is said to be a trivial object if $\emptyset\rightarrow W$ is a weak equivalence.
An object $A\in \mathcal{A}$ is said to be a cofibrant (resp. trivially cofibrant) if $\emptyset\rightarrow A$ is a cofibration (resp. trivially cofibration). Dually $B\in \mathcal{A}$ is fibrant (resp. trivially fibrant) if $B\rightarrow *$ is fibration (resp. trivial fibration). If $\CC, \CF$ and $\mathcal{W}$ are the classes of cofibrant, fibrant and weak equivalence objects in abelian model category $\mathcal{A}$, then $\mathcal{W}$ is a thick subcategory of $\mathcal{A}$ and both $(\mathcal{C},\mathcal{W}\cap\mathcal{F})$ and  $(\mathcal{C}\cap\mathcal{W},\mathcal{F})$ are complete cotorsion pairs in $\mathcal{A}$. Conversely, for classes of objects $\CC, \mathcal{W}$ and $\CF$, where $\mathcal{W}$ is thick, any two
complete cotorsion pairs of the above form give rise to an abelian model structure on $\mathcal{A}$, see \cite[Theorem 2.2]{Hov02}. These two cotorsion pairs have been referred to as Hovey pair. We also call $(\mathcal{C},\mathcal{W},\mathcal{F})$ a Hovey triple. 
 
\vspace{1mm}
We refer the reader to \cite{DS95} for a readable introduction to model categories and to \cite{Hov99}
for a more in-depth presentation. 

\subsection{Homotopy category of model category:}
Model categories are used to give an effective construction of the localization of categories, where the problem is to convert the class of weak-equivalence into isomorphisms. Suppose $\mathcal{C}$ is a category with a subcategory of $\mathcal{W}$. The localized category that denoted by $\mathcal{C}[\mathcal{W}^{-1}]$ is defined in classical algebra. In case $\mathcal{C}$ is a model category with weak equivalence $\mathcal{W}$, define $\mathcal{C}[\mathcal{W}^{-1}]$ as the \textit{Homotopy category} associated to $\mathcal{C}$ and denote by $\Ho \mathcal{C}$. Our reason for not adopting the right notation is that in this case, we have an identity between the morphisms of localized category and homotopy class of morphisms under a certain \textit{homotopy relation} which is determined by the model structure. The abstract notion of homotopy relation can be found in any references on model category such as \cite{Hov99}, but  whenever $\mathcal{A}$ is an abelian model category we can determine a homotopy relation by the following lemma:
\begin{lemma}
\label{lemma 27}
Let $\mathcal{A}$ be an abelian model category and $f,g:X\rightarrow Y$ be two morphisms. If $X$ is cofibrant and $Y$ is fibrant, then $f$ and $g$ are homotopic (we denote by $f\sim g)$ if and only if $f-g$ factor through a trivially fibrant and cofibrant object.
\end{lemma}
\begin{proof}
We refer to \cite[Proposition 4.4]{Gil11}.
\end{proof}

\section{Componentwise monoidal model structures}
\label{section 3}
Let $\mathcal{Q}$ be a quiver and $\mathcal{A}$ be an abelian category. In this section we develop a method for constructing some model structures on $C(Rep(\mathcal{Q},\mathcal{A}))$ where $C(Rep(\mathcal{Q},\mathcal{A}))$ is a category of all complex $\mathcal{X}^\bullet=(\mathcal{X}^i,d^i)$ which is $\mathcal{X}^i\in Rep(\mathcal{Q},\mathcal{A})$. We will use Hovey Theorem relating cotorsion pair to construct these model structures. 

Throughout this paper we use the following setting and notation:
\begin{notation}
\label{Not}
Let $\CQ=(V,E)$ be a quiver and $\mathcal{A}$ be an abelian category. Let $\CF$ be a
class of objects of $\mathcal{A}$ and contains the zero object of $\mathcal{A}$.
\begin{itemize}
\item[(a)] By $(\CQ,\mathcal{F})$ we mean the class of all representations $\mathcal{X}\in Rep(\mathcal{Q},\mathcal{A})$ such that $\mathcal{X}_v$ belongs to $\mathcal{F}$ for each vertex $v\in V$.
\item[(b)] By $C(\CQ,\mathcal{F})$ we mean the class of all complexes $\mathcal{X}^\bullet\in C(Rep(\mathcal{Q},\mathcal{A}))$ such that $\mathcal{X}^i$ belongs to $(\CQ,\mathcal{F})$ for each $i\in Z$.
\end{itemize}

\end{notation}
\begin{example}
\label{example}
\begin{itemize}
\item[(1)]
Let $\CQ=(V,E)$ be a quiver and $R$ be an associative ring with identity. Suppose that $C(R)$ denote the category of complexes over $R$ and $\CF=\Prj R$ is the class of projective $R$-module. By notation \ref{Not} $(\CQ,\Prj R)$ is equal to all representations $\CX$ in $\QR$ such that for each $v\in V$, $\CX_v$ is a projective $R$-module. So $C(\CQ,\Prj R)$ is the class of all complexes  in $C(\QR)$ such that each degree belongs to $(\CQ,\Prj R)$. If $\CX^\bullet \in C(\CQ,\Prj R)$, then $\CX^\bullet$ can be regarded as an object of $Rep(\mathcal{Q}, C(R))$ such that for each $v\in V$, $\CX_v^\bullet$ belongs to $C(\Prj R)$, where $C(\Prj R)$ is the class of all complexes of projective $R$-modules. Hence by notation \ref{Not}  we can say that $\CX^\bullet\in (\CQ, C(\Prj R))$, since $C(\Prj R)$ is a class of $C(R)$. Conversely, it is clear to see that every object of $(\CQ, C(\Prj R))$ can be regarded as an object of $C(\CQ,\Prj R)$,  hence these two categories have the same objects. 
\vspace{1mm}
\item[(2)]Let $\mathcal{Q}$ be the quiver $\xymatrix{ \bullet \ar[r] \ar@{.>}[r] & \bullet }$ and $\mathcal{F}= \Prj C(R)$. Consider the object $\mathcal{P}\in (\CQ,\Prj C(R))$ given as follows
$$\xymatrix{\cdots \ar[r] & 0 \ar[r] \ar[d] & Q \ar[r]^{1_Q} \ar[d]  & Q \ar[r] \ar[d]^{0}& 0 \ar[r] \ar[d] & 0 \ar[r] \ar[d] & \cdots \\ \cdots \ar[r] & 0 \ar[r] & 0 \ar[r] & P \ar[r]^{1_P} & P \ar[r] & 0 \ar[r] & \cdots  }$$ 
where $P$ and $Q$ are projective $R$-modules. 
Now if $\Prj C(\QR)$ is the class of all projective objects in $C(\QR)$, then $\mathcal{P}\notin \Prj C(\QR)$, since it is not a complex of projective representations. On the other hand $(\CQ,\Prj C(R))\subseteq (\CQ,C(\Prj R))$, since $\Prj C(R)\subseteq C(\Prj R)$. Hence by (1) we can say that $$\Prj C(\QR)\subset (\CQ,\Prj C(R))\subset (\CQ, C(\Prj R)).$$
\end{itemize}
\end{example}
\begin{remark}
As we see above, if $\CF$ is a class of objects of $\mathcal{A}$, then $(\CQ,C(\CF))$ and $C(\CQ,\CF)$ represent exactly the same class of $C(Rep(\mathcal{Q}, \mathcal{A}))$. In this paper we need both point of views. It is clear from the context that which one we have considered.
\end{remark}

In the following we need two lemmas. Carrying over the proof of \cite[Theorem A]{EHHS} and  \cite[Theorem 3.1]{EHHS} verbatim we have respectively:
\begin{lemma}
\label{lemma 31}
Let $\mathcal{Q}$ be an acyclic finite quiver and $\mathcal{A}$ be an abelian category which has
enough projective objects . Let $\CF$ be a class of objects of $\mathcal{A}$ and contains the zero object of $\mathcal{A}$. Then the pair $((\CQ,\CF),(\CQ,\CF)^\perp)$ (resp. $({}^\perp(\CQ,\CF),(\CQ,\CF))$) is a complete cotorsion pair if and only if the pair $(\mathcal{F},\mathcal{F}^\perp)$ (resp. $({}^\perp\mathcal{F},\mathcal{F})$) is so.
\end{lemma}
\begin{proof}
The proof is similar to the proof of Theorem \cite[Theorem A]{EHHS} by putting $(\CQ,\CF)=\xi$ and $\mathcal{F}=V_{\xi}$.
\end{proof}

\begin{lemma}
\label{lemma 32}
Let $\mathcal{Q}$ be an acyclic finite quiver and $\mathcal{A}$ be an abelian category which has
enough projective objects . Let $\CF$ be a class of objects of $\mathcal{A}$ and contains the zero object of $\mathcal{A}$. If $\mathcal{F}$ contain projective objects of $\mathcal{A}$, then $\mathcal{X}\in (\CQ,\CF)^\perp$ (resp. $\mathcal{X}\in {}^\perp(\CQ,\CF)$)
if and only if the following hold.
\begin{itemize}
 \item[$(i)$]  For any vertex $v$, $\mathcal{X}_v \in \mathcal{F}^\perp$ (resp. $\mathcal{X}_v \in {}^\perp \mathcal{F}$).
\item[$(ii)$]  For any vertex $v$, the map $\eta_{\CX,v}:\mathcal{X}_v \rightarrow \oplus_{s(a)=v}\mathcal{X}_{t(a)}$ (resp. $\xi_{\CX,v} : \oplus_{t(a)=v}\mathcal{X}_{s(a)}\rightarrow \mathcal{X}_v$) is an epimorphism (resp. a monomorphism) and $\Ker(\eta_{\CX,v})\in \mathcal{F}^\perp$ (resp. $\Coker(\xi_{\CX,v})\in {}^\perp \mathcal{F}$).
\end{itemize}
\end{lemma}
\begin{proof}
The proof is similar to the proof of Theorem \cite[Theorem 3.1]{EHHS} by putting $(\CQ,\CF)=\xi$ and $\mathcal{F}=V_{\xi}$.
\end{proof}
By lemma \ref{lemma 31} we immediately get the next result.
\begin{corollary}
\label{corollary 31}
Let $\mathcal{Q}$ be an acyclic finite quiver and $\mathcal{A}$ be an abelian category which has
enough projective objects. Let $(\mathcal{F},\mathcal{C})$ be a complete cotorsion pair in $C(\mathcal{A})$, then $((\CQ,\CF),(\CQ,\CF)^\perp)$ and $({}^\perp(\CQ,\CC),(\CQ,\CC))$ are complete cotorsion pairs.
\end{corollary}
\begin{proof}
Since $C(\mathcal{A})$ is an abelian category which has enough projective objects, therefore by using lemma \ref{lemma 31} we are done.  
\end{proof}

\begin{definition}
Let $\mathcal{A}$ be an abelian category. Suppose that $(\CA,\mathcal{B})$ and $(\mathcal{F},\mathcal{C})$ are two complete cotorsion pairs in $C(\mathcal{A})$. We say that they are compatible (or Hovey Pair) if $\mathcal{B}=\mathcal{C}\cap \mathcal{E}$ and $\mathcal{F}=\CA\cap \mathcal{E}$ where $\mathcal{E}$ is the class of all exact complexes in $C(\mathcal{A})$.
\end{definition}
\begin{theorem}
\label{prop 32}
Let $\mathcal{Q}$ be an acyclic finite quiver and $\mathcal{A}$ be an abelian category which has enough projective objects. Suppose that $(\CA,\mathcal{B})$ and $(\mathcal{F},\mathcal{C})$ is a Hovey pair in $C(\mathcal{A})$, then 
\begin{itemize}
\item[(a)]$((\CQ,\CA),(\CQ,\CA)^\perp)$ and $((\CQ,\CF),(\CQ,\CF)^\perp)$ is a Hovey pair in $C(Rep(\mathcal{Q},\mathcal{A}))$.
\item[(b)]$({}^\perp(\CQ,\mathcal{B}),(\CQ,\mathcal{B}))$ and $({}^\perp(\CQ,\CC)),(\CQ,\CC))$ is a Hovey pair in $C(Rep(\mathcal{Q},\mathcal{A}))$.
\end{itemize}
\end{theorem}
\begin{proof}
The two statements are dual. We will prove the first one. First of all by lemma \ref{lemma 31} $((\CQ,\CA),(\CQ,\CA)^\perp)$ and $((\CQ,\CF),(\CQ,\CF)^\perp)$ are complete cotorsion pairs. So we show that these cotorsion pairs are compatible. To this point we use lemma \ref{lemma 32}. Clearly $\mathcal{F}$ contains projective objects of $C(\mathcal{A})$. 
By assumption we have $\mathcal{B}=\mathcal{C} \cap \mathcal{E}$ and $\mathcal{F}=\CA \cap \mathcal{E}$. Let $\mathcal{E}_{\CQ}$ be the class of all exact complexes in $C(Rep(\CQ,\mathcal{A}))$. We have to show that $(\CQ,\CA)^\perp=(\CQ,\CF)^\perp \cap \mathcal{E}_{\CQ}$ and $(\CQ,\CF)=(\CQ,\CA) \cap \mathcal{E}_{\CQ}$. The second equality is trivial, since $\mathcal{F}=\CA \cap \mathcal{E}$.
\\
For the first equality first of all we note that $\CA^\perp=\mathcal{F}^\perp \cap \mathcal{E}$, since $\mathcal{B}=\mathcal{C} \cap \mathcal{E}$, $(\CA,\mathcal{B})$ and $(\mathcal{F},\mathcal{C})$ are cotorsion pair. Now let $\mathcal{X}^\bullet \in (\CQ,\CA)^\perp$. By Lemma \ref{lemma 32} $\mathcal{X}^\bullet$ satisfy in two conditions $(i)$ and $(ii)$, since $C(\mathcal{A})$ is an abelian category and $\CA$ contains projective objects of $C(\mathcal{A})$. Therefore for all $v \in V$ we have $\mathcal{X}_v^\bullet \in \CA^\perp$, $\eta_{\CX^\bullet,v}$ is an epimorphism and $ \Ker(\eta_{\CX^\bullet,v})\in \CA^\perp$. Therefore $\CX^\bullet_v, \Ker(\eta_{\CX^\bullet,v})\in \CF^\perp$, hence $\CX^\bullet\in (\CQ,\CF)^\perp$. On the other hand $\CX^\bullet\in \mathcal{E}_{\CQ}$, since $\CX_v^\bullet\in \CE$ for all $v\in V$. Therefore $\mathcal{X}^\bullet \in (\CQ,\CF)^\perp \cap \mathcal{E}_{\CQ}$.
\\
 Conversely, let $\CY^\bullet \in (\CQ,\CF)^\perp \cap \mathcal{E}_{\CQ}$. Since $\CY^\bullet\in (\CQ,\CF)^\perp$, hence by Lemma \ref{lemma 32} for all $v \in V$ we have $\mathcal{Y}_v^\bullet \in \mathcal{F}^\perp$, $\eta_{\CY^\bullet,v}$ is an epimorphism and $ \Ker(\eta_{\CY^\bullet,v})\in \mathcal{F}^\perp$. We also have $\Ker(\eta_{\CY^\bullet,v})\in \mathcal{E}$, since $\mathcal{Y}_v^\bullet \in \mathcal{E}$ for each vertex $v \in V$. Therefore $\Ker(\eta_{\CY^\bullet,v}),\CY_v\in \mathcal{F}^\perp\cap \CE=\CA^\perp$ for all $v\in V$. Now by Lemma \ref{lemma 32} we can say that $\CY^\bullet \in (\CQ,\CA)^\perp$. So we are done.  
\end{proof}
\begin{remark}
\label{Remark}
The above Theorem is a special case of \cite[Propositions 3.2 and 3.3]{D} but with different proof. Indeed, Let $\CQ$ be a left rooted quiver (so any acyclic finite quiver is left rooted) and let $(\CA,\mathcal{W},\CC)$ be a Hovey triple (or equivalently $(\CA,\mathcal{B})$ and $(\mathcal{F},\mathcal{C})$ is a Hovey pair, where $\mathcal{B}=\mathcal{W}\cap\CC$ and $\CF=\CA\cap \mathcal{W}$ in an abelian category $\mathcal{M}$ with coproducts and enough projectives. In \cite[Proposition 3.2]{D} it is proved that if $\mathcal{W}$ is closed under (small) coproducts then $(\CA,\mathcal{W},\CC)$ induces a certain Hovey triple
$(\Phi(\CA), \mathcal{T}, Rep(\CQ,\CC))$ in the abelian category $Rep(\CQ,\mathcal{M})$. Here by Lemma \ref{lemma 32} $\Phi(\CA)$ coincides with ${}^\perp(\CQ, \mathcal{B})$ and $Rep(\CQ,\CC)$ coincides  with $(\CQ,\CC)$. Thus, in our notation, the Hovey triple in $Rep(\CQ,\mathcal{M})$ produced by
\cite[Proposition 3.2]{D} is $({}^\perp(Q, \mathcal{B}), \mathcal{T}, (\CQ,\CC))$, equivalently,
$({}^\perp(Q, \mathcal{B}), (\CQ, \mathcal{B}))$ and $({}^\perp(Q, \mathcal{C}), (\CQ, \mathcal{C}))$
is a Hovey pair in $Rep(\CQ,\mathcal{M})$. Note that in case $\CQ$ is  sufficiently finite (In particular for acyclic and finite quiver) the assumption that $\mathcal{W}$ is closed under coproducts is not needed. Applying this to $\mathcal{M}=\CC(\mathcal{A})$ and using that $Rep(\CQ,C(\mathcal{A}))=C(Rep(\CQ,\mathcal{A}))$, one gets the Theorem \ref{prop 32}(b). Similarly, Theorem \ref{prop 32}(a) follows from \cite[Proposition 3.3]{D}. 
\end{remark}
\vspace{0.1cm}

Let $\CQ=(V,E)$ be an acyclic finite quiver and $\mathcal{A}$ be an abelian category. We will show that if we have a monoidal model structure on $C(\mathcal{A})$, then we can construct a monoidal model structure on $C(Rep(\mathcal{Q},\mathcal{A}))$. One of the reasons we are interested in monoidal category is that its homotopy category is also a symmetric monoidal category. We will remind the reader of the definition below; for more detail, see \cite[Chapter 4]{Hov99}.

In the category theory a symmetric monoidal category is a category $\mathcal{C}$ equipped with a functor  $\otimes:\mathcal{C}\times\mathcal{C}\longrightarrow \mathcal{C}$, called the tensor product, a unit object $S\in \mathcal{C}$, a natural associativity isomorphism $a_{X,Y,Z}:(X\otimes Y)\otimes Z \lrt X\otimes (Y\otimes Z)$, a natural left unit isomorphism $\lambda_X: S\otimes X \lrt X$, a natural right unit isomorphism $\rho_X: X\otimes S \lrt X$ and a natural isomorphism $B_{X,Y}: X\otimes Y \lrt Y\otimes X$ called the braiding, such that three coherence diagram commute. These coherence diagrams can be found in any references on category  such as \cite{ML71}.

A symmetric monoidal category $\mathcal{C}$ is closed if for all objects, $X\in \mathcal{C}$ the functor $-\otimes X: \mathcal{C}\lrt \mathcal{C}$ has a right adjoint functor.

Now suppose that $C(\mathcal{A})$ is a closed symmetric monoidal category equipped with the tensor product $\otimes$ and unit object $S$. In the following we will show that $C(Rep(\mathcal{Q},\mathcal{A}))$ is a closed symmetric monoidal category. To this point we define a new tensor product $\otimes_{cw}$ on $C(Rep(\mathcal{Q},\mathcal{A}))$. Let $\CX=(\CX_v,\varphi_a)_{v\in V, a\in E}$ and $\CY=(\CY_v,\psi_a)_{v\in V, a\in E}$ be two objects in $C(Rep(\mathcal{Q},\mathcal{A}))$. For each vertex $v\in V$, define $(\CX \otimes_{cw} \CY)_v= \CX_v \otimes \CY_v$, and for each arrow $a:v\lrt w$, define $\xymatrix{(\CX \otimes_{cw} \CY)_v \ar[r]^{\varphi_a \otimes_{cw}\psi_a} & (\CX \otimes_{cw} \CY)_w}$ by $\xymatrix{\CX_v\otimes \CY_v \ar[r]^{\varphi_a \otimes \psi_a} & \CX_w\otimes \CY_w}$. We also define the unit object $\mathcal{S}$ in $C(Rep(\mathcal{Q},\mathcal{A}))$ as follows:

For each vertex $v\in V$, set $\mathcal{S}_v= S$ and for each arrow $a:v\lrt w$, consider $\mathcal{S}_v \longrightarrow \mathcal{S}_w$ as a identity morphism.

It is straightforward to check that $(C(Rep(\mathcal{Q},\mathcal{A})),\otimes_{cw},\mathcal{S})$ is a symmetric monoidal category, since $(C(\mathcal{A}), \otimes, S)$ is so. Let $\CX$ be an arbitrary object in $C(Rep(\mathcal{Q},\mathcal{A}))$. Since the functor $-\otimes_{cw} \CX$ is right exact and preserves direct sums, it will have a right adjiont functor. Hence $(C(Rep(\mathcal{Q},\mathcal{A})),\otimes_{cw},\mathcal{S})$ is closed. 

Now suppose that we have an abelian model structure on $C(\mathcal{A})$ with $(\CA,\mathcal{B}), (\CF,\CC)$ as a Hovey pair. Hovey in \cite[Theorem 7.2]{Hov02} determine conditions on the functorially Hovey pair under which the resulting model structure will be compatible with the tensor product. To see that the model structure is monoidal ( with respect to the tensor product $\otimes$) we will prove the hypotheses of Hovey's Theorem 7.2. So we have the following theorem:
\begin{theorem}
Suppose that we have a monoidal model structure on $C(\mathcal{A})$ with respect to tensor product $\otimes$. Then there is a monoidal model structure on $C(Rep(\mathcal{Q},\mathcal{A}))$ with respect to tensor product $\otimes_{cw}$.
\end{theorem}
\begin{proof}
Let $(\CA,\mathcal{B})$, $(\CF,\CC)$ be a Hovey pair in $C(\mathcal{A})$. By Theorem \ref{prop 32}  $((\mathcal{Q},\CA),{(\mathcal{Q},\CA)}^\perp)$ and $((\mathcal{Q},\CF),{(\mathcal{Q},\CF)}^\perp)$ is a Hovey pair in $C(Rep(\mathcal{Q},\mathcal{A}))$. So the class of cofibrant object is equal to $(\mathcal{Q},\CA)$ and trivial cofibrant is equal to $(\mathcal{Q},\CF)$.
In view of \cite[Theorem 7.2]{Hov02} taking $\mathcal{P}$  to be the class of all short exact sequences in $C(Rep(\mathcal{Q},\mathcal{A}))$. Then we observe that Hovey's notion of $\mathcal{P}$-pure short exact sequence in this case just means a short exact sequence of complexes in $C(Rep(\mathcal{Q},\mathcal{A}))$ that is pure in each vertex $v\in V$. According to the Hovey's theorem it is easy to check that all conditions (a), (b), (c) and (d) satisfy with respect to $\otimes_{cw}$, since we have these condition for $\otimes$.
\end{proof}

\section{Two descriptions of Derived category }
\label{section 4}
Let $\mathcal{Q}$ be an acyclic finite quiver and $R$ be  an associative ring with identity. In this section we introduce some descriptions of the derived category of representations of $\mathcal{Q}$ in $\Mod R$. We write $D(\mathcal{Q})$ (resp. $K(\mathcal{Q})$, $C(\mathcal{Q})$) instead of $D(\QR)$ (resp. $K(\QR)$, $C(\QR)$). 

Let $\mathcal{E}$ be the class of exact complexes of $R$-modules. Recall that a complex $X^\bullet$ is DG-projective (DG-injective) if each $X^n$ is projective (resp. injective) and if $\mathcal{H}om(X^\bullet, E^\bullet)$ (resp. $\mathcal{H}om(E^\bullet, X^\bullet)$) is an exact complex for all $E^\bullet\in \mathcal{E}$. We denote by $dgPrj-R$ ($dgInj-R$) the class of all DG-projective (resp. DG-injective) complexes of $R$-modules. Since there is an equivalence between $\QR$ and $\Mod R\CQ$, we can define the concept of DG-Projective (DG-injective) complexes of representations of quiver $\CQ$ as the image of DG-projective (resp.  DG-injective) complexes of $R\CQ$-modules, under this equivalence. By \cite[Theorem 4.2.7]{ABHV}, a complex $\CX^\bullet\in C(\CQ)$ is a DG-projective if and only if for every vertex $v\in V$, $\CX^\bullet_v$ is a DG-projective complex in $C(R)$.
Throughout this section we use the following four different classes of representation of quiver $\mathcal{Q}$:

\begin{itemize}
\item[$\bullet$]${Prj}^{op}\mbox{-}\CQ=$ all representations $\CX\in\QR$ such that for every vertex $v$, $\CX_v$ is a projective module and the map $\eta_{\CX,v}:\mathcal{X}_v \rightarrow \oplus_{s(a)=v}\mathcal{X}_{t(a)}$ is split epimorphism.
\\
\item[$\bullet$]${Inj}^{op}\mbox{-}\CQ=$ all representations $\CX\in\QR$ such that for every vertex $v$, $\CX_v$ is injective module and the map $\xi_{\CX,v} : \oplus_{t(a)=v}\mathcal{X}_{s(a)}\rightarrow \mathcal{X}_v$ is split monomorphism.
\\
\item[$\bullet$]$dgPrj^{op}\mbox{-}\CQ=$ all representation  $\CX^\bullet\in Rep(\mathcal{Q},C(R))$ such that for every vertex $v$, $\CX^\bullet_v$ is DG-projective complexes of $R$-modules and the map $\eta_{\CX^\bullet,v}$ is split epimorphism.
\\
\item[$\bullet$]$dgInj^{op}\mbox{-}\CQ=$ all representation  $\CX^\bullet\in Rep(\mathcal{Q},C(R))$ such that for every vertex $v$, $\CX^\bullet_v$ is DG-injective complexes of $R$-modules and the map $\xi_{\CX^\bullet,v}$ is split monomorphism.
\end{itemize}

\vspace{0.2cm}
bf{Componentwise projective model structure:} Consider the well known Hovey pair $({dgPrj-}R,\CE)$ and $(Prj-C(R),C(R))$ in $C(R)$. Therefore by Theorem \ref{prop 32} we have the following
Hovey pair
$$\quad((\mathcal{Q},{dgPrj-}R),{(\mathcal{Q},{dgPrj-}R)}^\perp)\,\,\,\ , \,\,\,\,((\mathcal{Q}, Prj-C(R)),{(\mathcal{Q}, Prj-C(R))}^\perp)$$
in $C(\QR)$. 

By \cite[Theorem 2.2]{Hov02} we have an abelian model structure on $C(\CQ)$ which we call it{componentwise projective model structure}, where the weak equivalences are the homology isomorphisms, the cofibrations (resp. trivial cofibrations) are the monomorphisms with cokernels
in $(\CQ, {dgPrj-}R)$ (resp, $(\CQ,Prj-C(R)))$, and the fibrations (resp. trivial fibrations) are the
epimorphisms whose kernels are in $(\CQ,Prj-C(R))^\perp)$ (resp. $(\CQ, {dgPrj-}R)^\perp$).

Clearly the homotopy category of this model structure is equal to $D(\mathcal{Q})$. As we see in example \ref{example} this model structure is different than well known projective model structure on $C(\CQ)$. In the following we introduce some applications of componentwise projective model structureon $C(\CQ)$. First, we recall the notion of cofibrant replacement and fibrant replacement.
\begin{definition}
\label{def 41}
Let $\mathcal{C}$ be a model category. The axioms of model structure on $\mathcal{C}$ implies that any object $X\in \mathcal{C}$ has a cofibrant resolution consisting of cofibrant object $QX\in \mathcal{C}$ equipped with a trivially fibration $QX\lrt X$ in $\mathcal{C}$. Dually, $X$ has also a fibrant resolution consisting of a fibrant object $RX\in \mathcal{C}$ equipped with a trivially cofibration $X\lrt RX$. The object $QX$ (resp. $RX$) is called cofibrant replacement (resp. fibrant replacement) of $X$. 
\end{definition}

\begin{example}
\label{example 41}
Let $\mathcal{Q}$ be the quiver and $\CX^\bullet\in C(\mathcal{Q})$. Consider the componentwise projective model structure on $C(\mathcal{Q})$. We want to characterize a cofibrant replacement of $\CX^\bullet$. By definition \ref{def 41} if $Q\CX^\bullet$ is the cofibrant replacement of $\CX^\bullet$, then $ Q\CX^\bullet\in (\mathcal{Q},{dgPrj-}R)$ and $Q\mathcal{X}^\bullet \rightarrow \mathcal{X}^\bullet$ must be epimorphism such that $\Ker \rho\in (\mathcal{Q},{dgPrj-}R)^\perp$, i.e. for each vertex  $v\in V$, $(\Ker \rho)_v\in \CE$  and $\eta_{\ker \rho , v}$ is epimorphism. We construct $\mathcal{Q}\CX^\bullet$ in two steps.

{Step 1.} Since $({dgPrj-}R, \CE)$ is a complete cotorsion pair in $C(R)$, hence for each vertex $v\in V$ consider $P^\bullet$ as a DG-projective resolution of $\CX^\bullet_v$. By lifting property there exists $\mathcal{P}$ such that $\mathcal{P}\rightarrow  \mathcal{X}^\bullet$ is epimorphism. Now consider the short exact sequence $0 \rightarrow \mathcal{K} \rightarrow  \mathcal{P}\rightarrow \mathcal{X}^\bullet\rightarrow 0$. Clearly for each vertex $v\in V$, $\mathcal{K}_v\in \CE$ but $\eta_{\mathcal{K}, v}$ is not necessarily epimorphism.

 {Step 2.} Let $\mathcal{K}\rightarrow  \mathcal{P}$ be as follows
$$\xymatrix{ & K_2^\bullet \ar@{.>}[rr]^{\imath_2} & & P_2^\bullet \\
 K_1^\bullet \ar[ru]^{k_2} \ar[rd]_{k_3}   \ar@{.>}[rr]^{\imath_1} & & P_1^\bullet \ar[ru]^{p_2} \ar[rd]_{p_3} \\ 
 & K_3^\bullet \ar@{.>}[rr]^{\imath_3} & & P_3^\bullet}$$
Since  $\eta_{\mathcal{K},2}:K^\bullet_2\rightarrow 0$ and $\eta_{\mathcal{K},3}:K^\bullet_3\rightarrow 0$, hence in order to show that for each vertex $v\in V$, $\eta_{\mathcal{K}, v}$ is epimorphism we just need to focus on $\eta_{\mathcal{K},1}:K_1^\bullet \longrightarrow K_2^\bullet \oplus K_3^\bullet$. Consider a chain map $P^\bullet_{K_i} \rightarrow K_i^\bullet$ such that $P^\bullet_{K_i}$ is a projective complex and $\pi_i$ is an epimorphism, for $i=2,3$.  Clearly ${Cok}(\imath')$ is equal to $\CX^\bullet$ and $\eta_{\mathcal{K}',1}$ is epimorphism. But $\mathcal{P}'$ is a cofibrant object. Indeed we add projective complexes in vertex 1 of $\mathcal{P}$, so each vertex of $\mathcal{P}'$ is DG-projective complex.  Hence we introduce $\mathcal{P}'$ as a cofibrant replacement of $\mathcal{X}^\bullet$.
\end{example}
As we saw above in spite of ordinary projective model structure on $C(\CQ)$, the cofibrant replacement is obtained by considering a DG-projective resolution in each vertex and we do not care about arrows.
In the following we need more focus on the homotopy relation of this model structure.

Let $\mathcal{Q}$ be an acyclic finite quiver. Consider the componentwise projective model structure on $C(\mathcal{Q})$. Let $\CC$ (resp, $\CF$) be a class of cofibrant (resp. fibrant) objects and $\mathcal{W}$ be a class of trivial objects in this model structure. As we know above $\CC=(\mathcal{Q},{dgPrj-}R)$, $\mathcal{W}=\mathcal{E}_{\CQ}$ and $\CF={(\mathcal{Q}, \Prj C(R))}^\perp$. To understand the homotopy relation on this model structure we use lemma \ref{lemma 27}. By lemma \ref{lemma 31} we can say that $\mathcal{X}^\bullet$ is fibrant object if and only if for each vertex $v\in V$, $\eta_{\CX^\bullet,v}$ is an epimorphism. Furthermore the class $\mathcal{C}\cap \mathcal{W}=(\mathcal{Q},Prj-C(R))$. Hence the class $\mathcal{C}\cap \mathcal{W}\cap \CF$ is exactly equal to all objects $\mathcal{X}^\bullet\in C(\mathcal{Q})$ such that satisfy in the following conditions:
$$
 \begin{array}{ll}
(1)\,\,\mathcal{X}^\bullet_v\in \Prj C(R)\,\, {for each vertex} \,\, v\in V \\

(2)\,\, {For each vertex} \,\,\, v\in V, \eta_{\CX^\bullet,v}:\mathcal{X}_v^\bullet \rightarrow \bigoplus_{s(a)=v}\mathcal{X}_{t(a)}^\bullet \,\,\, {is epimorphism.}
\end{array}.
$$
So by lemma \ref{lemma 27} if $\mathcal{X}^\bullet$ is cofibrant object and $\mathcal{Y}^\bullet$ is fibrant object and $f,g:\mathcal{X}^\bullet\rightarrow\mathcal{Y}^\bullet$ then we say that $f$ and $g$ are homotopic, written $f\sim_{cw} g$, if and only if $f-g$ factor through an object $\mathcal{P}^\bullet$ such that satisfying two conditions above.

Next step, we will show the connection between this homotopy relation and ordinary homotopy relation in $C(\mathcal{Q})$.
\begin{lemma}
\label{lemma 41}
Let $\mathcal{Q}$ be an acyclic finite quiver. Consider componentwise projective model structure on $C(\mathcal{Q})$. If $f,g:\CX^\bullet \lrt \CY^\bullet$ are two morphisms of fibrant and cofibrant objects, then $f\sim_{cw} g$ if and only if $f\sim g$.
\end{lemma}

\begin{proof}
Suppose that $f\sim_{cw} g$. By assumption $f-g$ factor through an object $\mathcal{P}^\bullet$ such that satisfying two conditions above. Consider the Frobenius category $(C(\mathcal{Q}),\mathcal{S})$, where $\mathcal{S}$ is   the collection of short exact sequences in
$C(\mathcal{Q})$ of which each term is split short exact in $\QR$. We show that $\mathcal{P}^\bullet$ is an $\mathcal{S}$-injective object in this category. By carrying over the corresponding argument verbatim in proof of Theorem 4.2 in \cite{EER09}, we can say that $\mathcal{P}^\bullet$ is of the form $\bigoplus_{v\in v} e_{\rho,C(R)}^v(P_v^\bullet)$, where for any $v\in V$, $P_v^\bullet$ is the kernel of split epimorphism $\eta_{\mathcal{P}^\bullet,v}$. Since $P_v^\bullet \in \Prj C(R)$, hence $P_v^\bullet=\bigoplus_{i\in Z}\overline{P}[i]$. Therefore $$e_{\rho,C(R)}^v(P_v^\bullet)= e_{\rho,C(R)}^v(\bigoplus_{i\in Z}\overline{P}[i])= \bigoplus_{i\in Z}e_{\rho,C(R)}^v(\overline{P})[i]$$
It is easy to check that $e_{\rho,C(R)}^v(\overline{P})$ is the complex as follows
$$ \cdots \rightarrow 0 \rightarrow e_{\rho,R}^v(P) \rightarrow e_{\rho,R}^v(P) \rightarrow 0 \rightarrow \cdots $$
So we can say that $\bigoplus_{v\in V} e_{\rho,C(R)}^v(P_v^\bullet)$ is an $\mathcal{S}$-injective object in Frobenius category $(C(\mathcal{Q}),\mathcal{S})$, hence $f \sim g$.

Conversely, suppose that $f\sim g$. Therefore $f-g$ factors over an $\mathcal{S}$-injective object $I(\CX^\bullet)$, where $I(\CX^\bullet)$ is the complex $\oplus_{i\in Z} \overline{\CX^i}[i]$. By assumption $\CX^\bullet\in C(\mathcal{Q})_{cf}$, the full subcategory of cofibrant and fibrant objects of $C(\mathcal{Q})$, therefore for each $v\in V$, $\CX_v^\bullet\in {dgPrj-}R$ and $\eta_{\CX^\bullet,v}$ is split epimorphism. So $\CX^\bullet\in dgPrj^{op}\mbox{-}\CQ$. Hence we can say that for each $i\in Z$, $\CX^i\in {Prj}^{op}\mbox{-}\CQ $ and again in a similar manner of the proof of Theorem 4.2 in \cite{EER09}, we can say $\CX^i$ is of the form $\bigoplus_{v\in V} e_{\rho,R}^v(P^v)$ where $P^v$ is the kernel of split epimorphism of  $\eta_{\CX^i,v}$.  
So $\overline{\CX^i}$ is a direct sum of the complex as follows
$$ \cdots \rightarrow 0 \rightarrow e_{\rho,R}^v(P^v) \rightarrow e_{\rho,R}^v(P^v) \rightarrow 0 \rightarrow \cdots $$
Hence we can say that $\overline{\CX^i}=\bigoplus_{v\in V} e_{\rho,C(R)}^v(\overline{P^v})$. Now it is straightforward to check that $I(\CX^\bullet)$ satisfy in two conditions above, so $f\sim_{cw} g$. 
\end{proof}
\begin{theorem}
\label{theorem 43}
Let $\mathcal{Q}$ be an acyclic finite quiver. Then we have the following equivalence
$$K(dgPrj^{op}\mbox{-}\CQ)\cong D(\mathcal{Q})$$
\end{theorem}
\begin{proof}
Consider the componentwise projective model structure on $C(\mathcal{Q})$. As we see above $C(\mathcal{Q})_{cf}=dgPrj^{op}\mbox{-}\CQ $. By lemma \ref{lemma 41} we can say that:
$$ C(\mathcal{Q})_{cf}/\sim_{cw}= dgPrj^{op}\mbox{-}\CQ/\sim .$$
But we know that $dgPrj^{op}\mbox{-}\CQ/\sim= K(dgPrj^{op}\mbox{-}\CQ)$. On the other hand By \cite[Theorem 1.2.10]{Hov99} part (i) $C(\mathcal{Q})_{cf}/\sim_{cw}\cong D(\mathcal{Q})$.
So we are done.
\end{proof}
\begin{remark}
Note that in theorem above we introduce a subcategory, differ from subcategory of DG-projective complexes of $K(\mathcal{Q})$ such that equivalent to $D(\mathcal{Q})$ under the canonical functor $K(\mathcal{Q}) \longrightarrow D(\mathcal{Q})$.
\end{remark}
At the end of this section we will introduce another interpretation of derived category of complexes of representations of quivers  in terms of componentwise notion. The following basic result is folklore. We
provide here the argument for the reader's convenience.

\begin{proposition}
\label{prop 46} 
Let $\mathcal{A}$ be an abelian category with enough projectives. Let $(\CX,\CY)$ be a complete cotorsion pair in $C(\mathcal{A})$ such that every DG-projective complex belongs to $\CX$, equivalently, every object in $\CY$ is acyclic.
Then there is a triangulated equivalence
$$ K(\CX)/K_{\ac}(\CX)\cong D(\mathcal{A})$$
where $K(\CX)$ (resp. $K_{\ac}(\CX)$) is the full subcategory of $K(\mathcal{A})$  whose objects are (resp. acyclic) complexes those from $\CX$.
\end{proposition}
\begin{proof}
First we apply \cite[Lemma 4.7.1]{K} with $\mathcal{T}=K(\mathcal{A})$, $\mathcal{T}'=K(\CX)$, $\mathcal{S}=K_{ac}(\mathcal{A})$ and $\mathcal{S}'=\mathcal{S}\cap \mathcal{T}'=K_{\ac}(\CX)$ to see that the canonical triangulated functor
$$ J:K(\CX)/K_{\ac}(\CX)=\mathcal{T}'/\mathcal{S}'\longrightarrow \mathcal{T}/\mathcal{S}=K(\mathcal{A})/K_{ac}(\mathcal{A})=D(\mathcal{A}) $$
is fully faithful. Indeed, Let $X \rightarrow A$ be a morphism of complexes where $X$ is in $\CX$
and $A$ is acyclic. As $(\CX,\CY)$ is complete, there is an exact sequence of complexes $0 \rightarrow Y' \rightarrow X' \rightarrow A \rightarrow 0$ with $X' \in \CX$ and $Y'\in \CY$. As $A$ and $Y'$ are acyclic (by assumption on $\CY$), so is $X'$ and hence $X' \in K_{\ac}(\CX)$. As $X \rightarrow A$ factors through $X'$ (even in $C(\mathcal{A})$), it follows from condition (2) in \cite[Lemma 4.7.1]{K} that $J$ is fully faithful. It remains to see that $J$ is essentially surjective. Given
any complex $M$ take exact sequence of complexes $0 \rightarrow Y \rightarrow X \rightarrow M \rightarrow 0$
with $X\in \CX$ and $Y\in \CY$. Since $Y$ is, in particular, acyclic, the morphism
$X \rightarrow M$ is an isomorphism in $D(\mathcal{A})$, so $X \in K(\CX)$ with $J(X)\simeq
M$ in $D(\mathcal{A})$.
\end{proof}

\section*{Acknowledgments}

I would like to thank Rasool Hafezi for many useful hints and comments that improved the
exposition.

\end{document}